\documentclass[a4paper,11pt,reqno]{amsart}

\usepackage[margin=25mm]{geometry}
\usepackage[noadjust]{cite}
\usepackage[foot]{amsaddr}
\usepackage{hyperref}

\usepackage{dsfont,mathtools}
\usepackage[shortlabels]{enumitem}

\newcommand*{\abs}[1]{\left\lvert#1\right\rvert}
\newcommand*{\set}[1]{\left\{#1\right\}}
\newcommand{\prob}{\mathbf{P}}
\newcommand{\E}{\mathbf{E}}

\newcommand{\ND}{\mathcal N}
\newcommand{\real}{\mathbb R}
\newcommand{\dto}{\xrightarrow{d}}
\DeclareMathOperator{\Cov}{Cov}
\DeclareMathOperator{\hyper}{\sideset{_2}{_1}{\mathop F}}

\newtheorem{theorem}{Theorem}[section]
\newtheorem{lemma}[theorem]{Lemma}
\newtheorem{proposition}[theorem]{Proposition}
\theoremstyle{remark}
\newtheorem{remark}[theorem]{Remark}

\numberwithin{equation}{section}

\allowdisplaybreaks 

\begin{document}
\title[Parameter estimators in tempered fractional Vasicek model]{Asymptotic properties of parameter estimators in~Vasicek model driven by tempered fractional Brownian motion}

\author{Yuliya Mishura$^{1,2}$, Kostiantyn Ralchenko$^{1,3}$, and Olena Dehtiar$^1$}

\address{$^1$Taras Shevchenko National University of Kyiv, Ukraine}
\address{$^2$M\"alardalen University, V\"aster{\aa}s, Sweden}
\address{$^3$University of Vaasa, Finland}

\email{yuliyamishura@knu.ua, kostiantynralchenko@knu.ua, dehtiar.olena@knu.ua}

\thanks{The first  author is supported by the Swedish Foundation for Strategic Research, grant  no.\ UKR22-0017.}
\thanks{The second author is supported by the Research
Council of Finland, decision number 359815.}
\thanks{The first and the second authors acknowledge that the present research is carried out within the frame and support of the ToppForsk project no.\ 274410 of the Research Council of Norway with the title STORM: Stochastics for Time-Space Risk Models.}



\begin{abstract}
The paper focuses on the Vasicek model driven by a tempered fractional Brownian motion. We derive the asymptotic distributions of the least-squares estimators (based on continuous-time observations) for the unknown drift parameters. This work continues the investigation by Mishura and Ralchenko (Fractal and Fractional, 8(2:79), 2024), where these estimators were introduced and their strong consistency was proved.
\end{abstract}

\maketitle

\section{Introduction} 

The main goal of this paper is to establish the asymptotic distributions of the estimators of the drift parameters  in the tempered fractional Vasicek model, or, in other words, in the Vasicek model involving tempered fractional Brownian motion (TFBM) of the first kind, as the driver. This tempered process was introduced and studied   in \cite{TFBM, TFBM-integration}. Concerning the estimators, we use the estimators constructed in \cite{MR24}, where we established their strong consistency. Moreover, we applied the main results about asymptotic normality of the drift parameter's estimators in the Vasicek model with the Gaussian driver of the unspecified form, but satisfying several assumptions, proved in Theorem 3.2 of \cite{Es-Sebaiy}.
However, not all conditions of the specified theorem are satisfied for our model; therefore, direct application of  results from \cite{Es-Sebaiy} was impossible, and we had to significantly modify the main proofs. 
More information about the relation between our assumptions and those of \cite{Es-Sebaiy} is provided in Appendix \ref{app:es-sebaiy}.

The tempered fractional Vasicek model is described by the following stochastic differential equation:
\begin{equation}\label{eq:tfvm}
d Y_t = \left(a + b Y_{t}\right) d t + \sigma d B_{H,\lambda}(t), \quad t \geq 0, 
\quad Y_{0} = y_0,
\end{equation}
where $a\in\real$, $b>0$, $\sigma>0$, $y_0\in\real$ are constants, and $B_{H,\lambda} = \{B_{H,\lambda}, t\ge0\}$ is a tempered fractional Brownian motion introduced in \cite{TFBM}.

We focus on the case where $b > 0$ and continue to investigate the asymptotic behavior of the least squares estimator of the unknown parameters 
$(a,b)$. In \cite{MR24}, the strong consistency of the estimator was proved. In the present paper, we determine its asymptotic distribution. The main result indicates that, similar to the non-ergodic fractional Vasicek model studied in \cite{Es-Sebaiy}, the estimator of $a$ is asymptotically normal, whereas the estimator of 
$b$ follows a Cauchy-type asymptotic distribution. However, in the model \eqref{eq:tfvm}, the estimators of $a$ and $b$ are not asymptotically independent.

Our proofs rely on the asymptotic behavior of the variance of TFBM and the asymptotic growth with probability one of its sample paths. These properties were established in \cite{Azmoodeh} and \cite{MR24}, respectively.

Parameter estimation for a model similar to \eqref{eq:tfvm} but driven by fractional Brownian motion, known as the fractional Vasicek model, has been extensively studied for over 20 years. The case 
$a=0$, known as the fractional Ornstein--Uhlenbeck process, has been particularly well-studied. Drift parameter estimation for this case began in 2002 with the maximum likelihood estimation (MLE) discussed in \cite{kleptsynabreton2002, TudorViens}, and the asymptotic and exact distributions of the MLE were later investigated in \cite{tanaka2013, tanaka2015}.

Alternative approaches to drift parameter estimation for the fractional Ornstein--Uhlen\-beck process are found in \cite{rachidkhalifayoussef2011, yaozhong2010, HuNuZhou, kumirase}. Since the asymptotic behavior of this process and the estimators is significantly affected by the sign of the drift parameter, hypothesis testing methods for it were developed in \cite{Kukush17, Moers12}. For a comprehensive survey on drift parameter estimation in fractional diffusion models, see \cite{MR-survey}, and for a detailed presentation, we refer to the book \cite{KMR2017}.

Drift parameter estimation for Ornstein--Uhlenbeck processes driven by more general and related Gaussian processes was considered in \cite{yonghongjuan2021, mohamedkhalifayoussef2016, guangjunxiuweilitan2016, mendy2018, MRS23}. A similar problem for complex-valued fractional Ornstein--Uhlenbeck processes with fractional noise was investigated in \cite{yonghongjuanzhi2017}. In \cite{guangjunqian2019}, the least squares estimator for the drift of Ornstein--Uhlenbeck processes with small fractional L\'evy noise was constructed and studied.

In the general case of a fractional Vasicek model with two unknown drift parameters, the least squares and ergodic-type estimators were studied in \cite{Lohvinenko-Ralchenko-Zhuchenko-2016, weilinjun2018, weilinjun2019}, while the corresponding MLEs were investigated in \cite{Lohvinenko-Ralchenko-2017, Lohvinenko-Ralchenko-2018, Lohvinenko-Ralchenko-2019, tanaka2020maximum}. In \cite{shengfengji2018}, the least squares estimators of the Vasicek-type model driven by sub-fractional Brownian motion were studied. The same problem for the case of more general Gaussian noise (including fractional, sub-fractional, and bifractional Brownian motions) was investigated in \cite{Es-Sebaiy}. Least squares estimation of the drift parameters for the approximate fractional Vasicek process was investigated in \cite{jixiaxiofangchao2023}.

Several papers are devoted to the model \eqref{eq:tfvm} with non-Gaussian noises. In particular, drift parameter estimation for a Vasicek model driven by a Hermite process was studied in \cite{nourdindiu2017}; Vasicek-type models with L\'evy processes were considered in \cite{es-sebaiy_al-foraih_alazemi_2021, reiichiro}.

It is worth mentioning that the theory of parameter estimation for stochastic differential equations driven by a standard Wiener process, especially for classical Ornstein--Uhlenbeck and Vasicek models, is now well-developed. For comprehensive resources, see the books \cite{Bishwal, Iacus, Kutoyants, Liptser-stat2}. More recent results in this direction can be found in \cite{KMR2017} and the papers \cite{huixing2015, huihuizhou2020, Prykhodko24, shimizu2010, TangChen}. Additionally, parameter estimation for the reflected Ornstein--Uhlenbeck process was studied in \cite{zangzhang2019}, and for the threshold Ornstein--Uhlenbeck process in \cite{jaozhongjuejan2022}.

The structure of this paper is as follows.
In the beginning of Section~~\ref{sec:main}, we recall the definition and properties of the TFBM. Subsequently, we introduce the tempered fractional Vasicek model and the least-squares-type estimators for the drift parameters, and we formulate the main result concerning the asymptotic distributions of these estimators.
All proofs are provided in Section~\ref{sec:proofs}.
In Subsection~\ref{ssec:repres}, we express our estimators in terms of three Gaussian processes, which are three different integrals involving TFBM.
Next, in Subsection~\ref{ssec:as-norm-aux}, we determine the joint asymptotic distribution of these Gaussian processes.
This enables us to derive the proof of the main theorem, which is detailed in Subsection~\ref{ssec:proof-main}.
The paper is supplemented with two appendices.
In Appendix~\ref{app:es-sebaiy}, we discuss the relation between our model and the conditions presented in \cite{Es-Sebaiy}.
Appendix~\ref{app:special} provides brief information on the special functions that arise in the calculation of the asymptotic variances of the estimators.

\section{Model description and main result}
\label{sec:main}

\subsection{Tempered fractional Brownian motion}
Let $W = \set{W_x, x\in\real}$ be a two-sided Wiener process, $H>0$, $\lambda>0$.
According to \cite{TFBM}, a tempered fractional Brownian motion (TFBM) is a zero mean stochastic process $B_{H,\lambda}= \{B_{H,\lambda}(t), t\ge0\}$ defined by the following Wiener integral
\begin{equation*}
B_{H,\lambda}(t) = \int_{\real} \left[\exp\set{-\lambda  (t-x)_{+}}(t-x)^{H-\frac{1}{2}}_{+}-\exp\set{-\lambda(-x)_{+}}(-x)^{H-\frac{1}{2}}_{+}\right]dW_x.
\end{equation*}
Its covariance function has the following form \cite{TFBM}
\begin{equation}\label{eq:cov}
\Cov [B_{H,\lambda}(t), B_{H,\lambda}(s)]
= \frac{1}{2} \left(C_{t}^2 t^{2H}+ C_{s}^2 s^{2H} - C_{\abs{t-s}}^2 \abs {t-s}^{2H}\right),
\end{equation}
with
\begin{equation}\label{eq:Ct}
C_t^2 = \frac{2\Gamma(2H)}{(2\lambda t)^{2H}}-\frac{2\Gamma(H+\frac{1}{2})}{\sqrt{\pi}\,(2\lambda t)^{H}} K_H(\lambda t),
\end{equation}
where $K_\nu(z)$ is the modified Bessel function of the second kind, see Appendix \ref{app:special}.

The variance function of TFBM with parameters $H > 0$ and $\lambda > 0$ satisfies
\begin{equation}\label{eq:TFBMI-asymp}
\lim_{t\to+\infty} \E \left[B_{H,\lambda}(t)\right]^2 =
\lim_{t\to+\infty} C_{t}^2 t^{2H} =
\frac{2\Gamma(2H)}{(2\lambda)^{2H}} \eqqcolon \alpha_{H,\lambda}^2,
\end{equation}
see \cite[Proposition 2.4]{Azmoodeh}.

Furthermore, it was proved in \cite[Theorem 3]{MR24} that
for any $\delta>0$, there exists a non-negative random variable $\xi = \xi(\delta)$ such that for all $t>0$
\begin{equation}\label{eq:tfbm-growth}
\sup_{s\in[0,t]}\abs{B_{H,\lambda}(s)}\le \left(t^\delta \vee 1\right)\xi
\quad\text{a.s.},
\end{equation}
and there exist positive constants 
$C_1 = C_1(\delta)$ and $C_2 = C_2(\delta)$
such that for all $u>0$
\begin{equation*}
\prob(\xi>u) \le C_1 e^{-C_2 u^2}.
\end{equation*}

\subsection{Parameter estimation in the tempered fractional Vasicek model}
We focus in this paper on the drift parameter estimation for the tempered fractional Vasicek model, which is described by the following stochastic differential equation:
\begin{equation}\label{Main-SDE}
Y_t = y_0 + \int_0^t (a + b Y_s)\,ds + \sigma B_{H,\lambda}(t),
\quad t > 0,
\quad Y_0 = y_0,
\end{equation}
where
$a\in\real$, $b>0$, $\sigma>0$, $y_0\in\real$.
The solution $Y = \{Y_t, t\ge0\}$ is given explicitly by
\begin{equation}\label{eq:SDE-solution}
Y_t = \left(y_0 + \frac{a}{b}\right) e^{bt} - \frac{a}{b} + \sigma \int_0^t e^{b(t-s)}dB_{H,\lambda}(s),
\end{equation}
where 
the integral is defined by the integration by parts:
\begin{equation}\label{eq:int}
\int_0^t e^{b(t-s)}dB_{H,\lambda}(s)
\coloneqq
B_{H,\lambda}(t) + b \int_0^t e^{b(t-s)} B_{H,\lambda}(s) ds.
\end{equation}

Let us consider the estimation of unknown drift parameter $\theta = (a,b) \in\real\times(0,\infty)$ in the model \eqref{Main-SDE}.
Following \cite{MR24}, we define the estimator $\hat\theta_T = (\hat a_T, \hat b_T)$ as follows
\begin{gather}
\hat a_T = \frac{\left(Y_T - y_0\right) \left(\int_0^T Y_t^2 dt - \frac12 \left(Y_T+y_0\right) \int_0^T Y_t dt\right)}{T \int_0^T Y_t^2 dt - \left(\int_0^T Y_t dt\right)^2},
\label{eq:a-est}
\\
\hat b_T = \frac{\left(Y_T - y_0\right) \left(\frac12 T \left(Y_T + y_0\right)  - \int_0^T Y_t dt\right)}{T \int_0^T Y_t^2 dt - \left(\int_0^T Y_t dt\right)^2}.
\label{eq:b-est}
\end{gather}
According to 
\cite[Theorem 6]{MR24},
$(\hat a_T, \hat b_T)$ is a strongly consistent estimator of the parameter $(a,b)$ as $T\to\infty$.
The purpose of the present paper is to find asymptotic distributions of $\hat a_T$ and $\hat b_T$. More precisely, we shall prove that $\hat a_T$ is asymptotically normal, and $\hat b_T$ has asymptotic Cauchy-type distribution.

\subsection{Main result}
Let us introduce the notations 
\begin{equation}\label{eq:alphabeta}
\alpha_{H,\lambda}^2 = \frac{2\Gamma(2H)}{(2\lambda)^{2H}}
\quad \text{and} \quad
\beta_{H,\lambda,b}^2 = \frac{b}{2} \int_{0}^{\infty} \exp\{-b u \} C^2_{u} u^{2H}du.
\end{equation}
The following theorem is the main result of the paper. 

\begin{theorem}\label{th:main}

The   estimators $\hat a_T$ and $\hat b_T$ from \eqref{eq:a-est} and 
\eqref{eq:b-est}, respectively, have the following asymptotic properties. 
\begin{enumerate}[(i)]
\item
The estimator $\hat a_T$ is asymptotically normal:
\begin{equation}\label{eq:a-asnorm}
T\left(\hat a_T - a\right) \dto 
\ND \left( 0, \sigma^2 \alpha_{H,\lambda}^2 \right)
\quad
\text{as } T\to\infty.
\end{equation}

\item
The estimator $\hat b_T$ has asymptotic Cauchy-type distribution:
\[
e^{b T}\left(\hat b_T - b\right) \dto \frac{\eta_1}{\eta_2},
\quad
\text{as } T\to\infty,
\]
where
$\eta_1 \simeq \ND(0, 4 b^2 \sigma^2 \beta_{H,\lambda,b}^2)$
and
$\eta_2 \simeq \ND(y_0 + \frac{a}{b}, \sigma^2 \beta_{H,\lambda,b}^2)$
are independent normal random variables.
\end{enumerate}

\end{theorem}

\begin{remark}[Joint distribution of the estimators]
\label{rem:joint}
Unlike the case of the Vasicek model driven by fractional Brownian motion (see \cite[Proposition 4.1]{Es-Sebaiy}), the estimators $\hat a_T$ and $\hat b_T$ are \emph{not} asymptotically independent.
More precisely, the following convergence holds
\begin{equation}\label{eq:joint}
\begin{pmatrix}
T\left(\hat a_T - a\right)
\\
e^{b T}\left(\hat b_T - b\right)
\end{pmatrix}
\dto
\begin{pmatrix}
b\sigma \xi_2 - \sigma \xi_3
\\
\frac{2b\sigma\xi_3}{y_0 + \frac{a}{b} + b\sigma\xi_1}
\end{pmatrix},
\end{equation}
where the random vector $(\xi_1, \xi_2, \xi_3)$ has a Gaussian distribution $\ND (\mathbf 0, \Sigma)$ with the covariance matrix $\Sigma$ defined in Proposition \ref{prop:ZUV-asymp} below.
We see that the normal random variables $\xi_1 \simeq \ND (0, b^{-2}\beta_{H,\lambda,b}^2)$ and $\xi_3 \simeq \ND (0,\beta_{H,\lambda,b}^2)$ are independent, and so are $\xi_2 \simeq \ND (0, b^{-2}( \alpha_{H,\lambda}^2-\beta_{H,\lambda,b}^2))$ and $\xi_3$.
However, there is a correlation between $\xi_1$ and $\xi_2$, namely
$\Cov(\xi_1, \xi_2) = b^{-2} \beta_{H,\lambda,b}^2$.
\end{remark}

\begin{remark}[Representation for $\beta_{H,\lambda,b}^2$ via hypergeometric function]
The constant $\beta_{H,\lambda,b}^2$ can be represented in an alternative form, which may be more suitable for its numerical computation.
Using formula \eqref{eq:Ct} for $C_t$, we can rewrite it in the following form
\begin{align}
\beta_{H,\lambda,b}^2 &= \frac{b}{2} \int_{0}^{\infty} \exp\{-b t \} \left(\frac{2\Gamma(2H)}{(2\lambda)^{2H}}-\frac{2\Gamma(H+\frac{1}{2}) t^H}{\sqrt{\pi}\,(2\lambda)^{H}} K_H(\lambda t)\right) dt
\notag\\
&= \frac{\Gamma(2H)}{(2\lambda)^{2H}} 
- \frac{b \Gamma(H+\frac{1}{2})}{\sqrt{\pi}\,(2\lambda)^{H}} \int_{0}^{\infty} \exp\{-b t \}  t^H K_H(\lambda t) dt.
\label{eq:beta2}
\end{align}
Furthermore, by \cite[formula 6.621-3]{Gradshteyn},
\begin{multline}\label{eq:beta3}
\int_{0}^{\infty} \exp\{-b t \}  t^H K_H(\lambda t) dt
\\*
= \frac{\sqrt{\pi} (2\lambda)^H}{(b + \lambda)^{2H + 1}}
\,\frac{\Gamma(2H + 1)}{\Gamma(H+\frac32)}
\hyper\left(2H+1, H+\tfrac12; H+\tfrac32; \frac{b - \lambda}{b + \lambda}\right),
\end{multline}
where $\hyper$ denotes the Gauss hypergeometric function, see Appendix \ref{app:special}.
Hence, combining \eqref{eq:beta2}--\eqref{eq:beta3} and using the relation $\Gamma(H+\frac32)= (H+\frac12) \Gamma (H+\frac12)$,
we arrive at
\begin{equation}\label{eq:repr-beta-hyp}
\begin{split}
\beta_{H,\lambda,b}^2 &= \frac{\Gamma(2H)}{(2\lambda)^{2H}} 
- \frac{2b \Gamma(2H+1)}{(b + \lambda)^{2H + 1}(2H+1)}
\hyper\left(2H+1, H+\tfrac12; H+\tfrac32; \frac{b - \lambda}{b + \lambda}\right)
\\
&= \frac{1}{2} \alpha_{H,\lambda}^2
- \frac{2b \Gamma(2H+1)}{(b + \lambda)^{2H + 1}(2H+1)}
\hyper\left(2H+1, H+\tfrac12; H+\tfrac32; \frac{b - \lambda}{b + \lambda}\right).
\end{split}
\end{equation}
\end{remark}

\section{Proofs}
\label{sec:proofs}
Let us introduce the following processes
\begin{gather}
Z_t \coloneqq \int_0^t e^{-bs} B_{H,\lambda}(s)ds,
\quad
U_t = e^{-bt} \int_0^t e^{bs} B_{H,\lambda}(s) ds,
\label{eq:ZtUt}
\\
V_t \coloneqq e^{-bt} \int_0^t e^{bs} dB_{H,\lambda}(s)
= B_{H,\lambda}(t) - b U_t.
\label{eq:Vt}
\end{gather}
The proof of the main result will be conducted according to the following scheme.
First, in subsection \ref{ssec:repres} we express $T(\hat a_T - a)$ and $e^{b T}(\hat b_T - b) $ via the processes $Z$, $U$ and $V$ and remainder terms, vanishing at infinity.
Then in subsection \ref{ssec:as-norm-aux} we find the joint asymptotic distribution of the Gaussian vector $(Z_T, U_T, V_T)$ as $T \to \infty$. Finally, using these results along with the Slutsky theorem, we derive the limits in distribution for $T(\hat a_T - a)$ and $e^{b T}(\hat b_T - b)$ as $T \to \infty$ in subsection~\ref{ssec:proof-main}.

\subsection{Representation of the estimators}
\label{ssec:repres}

Let us recall some well-known facts about the convergence of integrals involving tempered fractional Vasicek process $Y$.
It was proved in \cite{MR24} that the random variable
\begin{equation}\label{eq:Zinf}
Z_\infty \coloneqq \int_0^\infty e^{-bs} B_{H,\lambda}(s)\, ds
\end{equation}
is well defined and
the following convergences hold a.s.\ as $T\to\infty$:
\begin{gather}\label{eq:Y-asymp}
e^{-bT} Y_T \to \zeta,
\\
\label{eq:int-Y-asymp}
e^{-bT} \int_0^T Y_t dt \to \frac{1}{b} \zeta,
\\
\label{eq:int-Y2-asymp}
e^{-2bT} \int_0^T Y_t^2 dt \to \frac{1}{2b} \zeta^2,
\\
\label{eq:int-Yt-asymp}
T^{-1}e^{-bT} \int_0^T Y_t t\, dt \to \frac{1}{b} \zeta,
\end{gather} 
where
\begin{equation}\label{eq:zeta}
\zeta \coloneqq y_0 + \frac{a}{b} + b \sigma Z_\infty,
\end{equation}
see \cite[Lemma 6]{MR24}.

Now we are ready to formulate and prove an auxiliary lemma, which is crucial for the proof of the main theorem.
The lemma provides a representation of the estimator $\hat b_T$ via the integrals $Z_T$ and $V_T$ defined in \eqref{eq:ZtUt}--\eqref{eq:Vt}.
\begin{lemma}\label{l:delta-b}
For all $T>0$
\begin{equation}\label{eq:repr-b}
e^{bT} (\hat b_T - b) = \frac{\sigma V_T \left(y_0 + \frac{a}{b} + b\sigma Z_T\right)}{D_T} +  R_T,
\end{equation}
where
\begin{equation}
\label{eq:Dt}
D_T \coloneqq e^{-2bT} \left(\int_0^T Y_t^2 dt - \frac1T \left(\int_0^T Y_t dt\right)^2\right)
\to \frac{1}{2b} \zeta^2
\quad \text{a.s., as } T \to\infty,
\end{equation}
and
\begin{equation}\label{eq:RT}
R_T \to 0 \quad \text{a.s., as } T \to\infty.
\end{equation}
\end{lemma}

\begin{proof}
By the definition \eqref{eq:b-est} of the estimator $\hat b_T$,  
\begin{equation}
\label{eq:int-b-asymp}
e^{bT}(\hat b_T - b) = \frac{F_T}{D_T},
\end{equation}
where the denominator $D_T$ is defined by \eqref{eq:Dt}, and the
the numerator $F_T$ has the following form
\begin{align}
F_T &= e^{-bT}(Y_T-y_0)\left(\frac{1}{2}(Y_T+y_0)-\frac{1}{T}\int_0^T Y_t dt\right)
- be^{-bT}\left(\int_0^T Y^2_t dt - \frac{1}{T}\left(\int_0^T Y_t dt\right)^2\right)
\notag\\
&= F_{1,T}+F_{2,T}+F_{3,T}+F_{4,T},
\label{eq:FT}
\end{align}
where
\begin{align*}
F_{1,T} &= \frac{1}{2}e^{-bT}(Y_T-y_0)(Y_T+y_0),
&
F_{2,T} &= -\frac{1}{T}e^{-bT}(Y_T-Y_0)\int_0^T Y_t dt,
\\
F_{3,T} &= -be^{-bT}\int_0^T Y^2_t dt,
&
F_{4,T} &= \frac{b}{T}e^{-bT}\left(\int_0^T Y^2_t dt\right)^2.
\end{align*}
Let us consider each of $F_{i,T}$ separately.
Substituting the right-hand side of the equation \eqref{Main-SDE} instead of the process $Y$, we rewrite the term $F_{1,T}$ as follows:
\begin{align}
F_{1,T} &= \frac{1}{2}e^{-bT}\left(aT + b\int_0^T Y_t dt + \sigma B_{H,\lambda}(T)\right)\left(2y_0+aT+b\int_0^T Y_t dt + \sigma B_{H,\lambda}(T)\right)
\notag\\
&=abTe^{-bT}\int_0^T Y_t dt + y_0be^{-bT}\int_0^T Y_t dt + \frac{1}{2}b^2e^{-bT}\left(\int_0^T Y_t dt\right)^2
\notag\\
&\quad+b\sigma e^{-bT} B_{H,\lambda}(T) \int_0^T Y_t dt + \frac{1}{2}e^{-bT} (aT + \sigma B_{H, \lambda}(T))(2y_0 + aT + \sigma B_{H,\lambda}(T)).
\label{eq:FT1_re}
\end{align}
Note that it follows from \eqref{eq:SDE-solution}, \eqref{eq:int}, and \eqref{eq:ZtUt} that the process Y has the following representation:
\begin{equation}\label{rep:YT}
Y_T = \left(y_0 + \frac{a}{b}\right)e^{bT}-\frac{a}{b}+b\sigma e^{bT}Z_T + \sigma B_{H,\lambda}(T).
\end{equation}
Moreover, expressing $b\int_0^T Y_t dt$ from the equation \eqref{Main-SDE} and using \eqref{rep:YT}, we get
\begin{equation}\label{eq:bYt}
b\int_0^T Y_tdt = Y_T - y_0 -aT-\sigma B_{H,\lambda}(T) = \left(y_0+\frac{a}{b}\right)e^{bT}
+b\sigma e^{bT}Z_T - y_0 -\frac{a}{b}-aT.
\end{equation}
Now we insert \eqref{eq:bYt} into the fourth term in the right-hand side of \eqref{eq:FT1_re} and obtain
\begin{multline}\label{eq:FT1_re2}
F_{1,T} = abTe^{-bT}\int_0^T Y_t dt + y_0be^{-bT}\int_0^T Y_t dt + \sigma \left(y_0+\frac{a}{b}\right) B_{H, \lambda}(T)\\*
+ b\sigma^2 Z_T B_{H, \lambda}(T)+R_{1,T},
\end{multline}
where
\begin{equation*}
R_{1,T} = \frac{1}{2}e^{-bT}(aT+\sigma B_{H,\lambda}(T))(2y_0+aT+\sigma B_{H,\lambda}(T))-\sigma e^{-bT}\left(y_0+\frac{a}{b}+aT\right)B_{H,\lambda}(T).
\end{equation*}
In view of \eqref{eq:tfbm-growth}
\begin{equation}\label{eq:RT2}
R_{1,T} \to 0 \quad \text{a.s., as } T \to\infty.
\end{equation}
Let us consider $F_{2,T}$. By \eqref{Main-SDE},
\begin{align}
F_{2,T} &= -\frac{1}{T}e^{-bT} \left(aT+b\int_0^T Y_t dt + \sigma B_{H,\lambda}(T)\right)\int_0^T Y_t dt
\notag\\
&=-ae^{-bT}\int_0^T Y_t dt - \frac{b}{T}e^{-bT}\left(\int_0^T Y_t dt\right)^2 + R_{2,T},
\label{eq:FT2_re2}
\end{align}
where
\begin{equation}\label{eq:RT2_2}
R_{2,T} = -\frac{\sigma}{T}e^{-bT}B_{H, \lambda}(T) \int_0^T Y_t dt \to 0 \quad \text{a.s., as } T \to\infty,
\end{equation}
due to \eqref{eq:tfbm-growth} and \eqref{eq:int-Y-asymp}.

Further, we transform $B_{3,T}$ using \eqref{Main-SDE} as follows:
\begin{align}
F_{3,T} &= -be^{-bT}\int_0^T Y_t \left(y_0 + at + b\int_0^tY_s ds + \sigma B_{H,\lambda}(t)\right)dt 
\notag\\
&= -by_0e^{-bT}\int_0^T Y_t dt - abe^{-bT}\int_0^T tY_t dt-b^2e^{-bT}\int_0^T Y_t\int_0^t Y_s \,ds\,dt
\notag\\
&\quad- b\sigma e^{-bT}\int_0^T Y_tB_{H,\lambda}(t)dt
\notag\\
&\coloneqq F_{31,T}+F_{32,T}+F_{33,T}+F_{34,T}.
\label{eq:FT3_re2}
\end{align}
Integrating by parts and applying \eqref{eq:bYt} we get 
\begin{align}
F_{32,T} &= -abe^{-bT}\int_0^T t \, d\left(\int_0^tY_s ds\right)
= -ab Te^{-bT} \int_0^T Y_t dt + abe^{-bT}\int_0^T \!\!\int_0^t Y_s ds\,dt 
\notag\\
&= -abTe^{-bT}\int_0^T Y_t dt + ae^{-bt} \int_0^T(Y_t - y_0-at-\sigma B_{H,\lambda}(t))dt
\notag\\
&= -ab Te^{-bT}\int_0^T Y_t dt + ae^{-bT}\int_0^T Y_t dt + R'_{3,T},
\label{eq:FT32_re2}
\end{align}
where 
\begin{equation}\label{eq:RT3_re2}
R'_{3,T} = -ae^{-bT} \int_0^T (y_0 + at + \sigma B_{H,\lambda}(t)) dt \to 0 \quad \text{a.s., as } T \to\infty
\end{equation}
by \eqref{eq:tfbm-growth}.

Due to symmetry of the integrand, it is not hard to see that $\int_0^T\!\!\int_0^t Y_t Y_s dsdt = \frac{1}{2}(\int_0^T Y_t dt)^2$, whence 
\begin{equation}\label{eq:F33_re2}
F_{33,T} = -\frac{b^2}{2}e^{-bT}\left(\int_0^T Y_t dt\right)^2.
\end{equation}

In order to transform $F_{34,T}$, we use \eqref{eq:bYt} and get
\begin{align*}
F_{34,T} &= -b\sigma e^{-bT}\int_0^T B_{H, \lambda}(t) \left(\left(y_0 +\frac{a}{b}\right)e^{bt}-\frac{a}{b}+b\sigma e^{bt}Z_t+\sigma B_{H,\lambda}(t)\right)
\notag\\
&= - b\sigma(y_0 + \frac{a}{b})e^{-bT}\int_0^T e^{bT} B_{H,\lambda}(t)dt + a\sigma e^{-bT}\int_0^T B_{H,\lambda}(t)dt 
\notag\\
&\quad- b^2\sigma^2e^{-bT}\int_0^Te^{bt}B_{H,\lambda}(t)Z_t dt - b\sigma^2 e^{-bT}\int_0^T B^2_{H,\lambda}(t)dt.
\end{align*}
Using integration by parts, we obtain 
\begin{align*}
\int_0^T e^{bt}B_{H,\lambda}(t)Z_t dt &= \int_0^T Z_t d \left(\int_0^t e^{bs}B_{H,\lambda}(s)ds\right) \\
&= Z_T \int_0^T e^{bt} B_{H,\lambda}(t)dt - \int_0^T\!\!\int_0^te^bs B_{H,\lambda}(s)dsdZ_t
\end{align*}
Hence,
\begin{align*}
F_{34,T} &= -b\sigma \left(y_0 +\frac{a}{b}\right)e^{bt}\int_0^T e^{bt}B_{H,\lambda}(t)dt - b^2\sigma^2e^{-bT}Z_T\int_0^Te^{bt}B_{H,\lambda}(t)dt+R''_{3,T},
\end{align*}
where 
\begin{multline}\label{eq:R''3T_re2}
R''_{3,T} = a\sigma e^{-bT} \int_0^T B_{H,\lambda}(t)dt - b^2\sigma^2e^{-bT}\int_0^T B^2_{H,\lambda}(t)dt
\\
+b^2\sigma^2e^{-bT}\int_0^T \!\!\int_0^t e^{bs} B_{H,\lambda}(s)ds dZ_t.
\end{multline}

Recall that by \eqref{eq:ZtUt}--\eqref{eq:Vt}
\begin{align*}
e^{-bT}\int_0^T e^{bt} B_{H,\lambda}(t) dt = U_T = \frac{1}{b} B_{H, \lambda} (T) - \frac{1}{b}V_T.
\end{align*}

Therefore, we can rewrite $F_{34, T}$ as follows:
\begin{align}\label{eq:F34T_re2}
F_{34,T} &= -\sigma \left(y_0 +\frac{a}{b}\right)B_{H,\lambda}(T) - b\sigma^2 B_{H,\lambda}(T)Z_T + \sigma V_T(y_0 +\frac{a}{b}+b\sigma Z_T)+R''_{3,T}.
\end{align}

Note that the first two terms in the right-hand side of \eqref{eq:R''3T_re2} converge to zero a.s., as  $T \to\infty$ due to \eqref{eq:tfbm-growth}. The last term in \eqref{eq:R''3T_re2} also vanishes, because 
\begin{align*}
\lim_{T\to\infty} e^{-bT}\int_0^T\!\! \int_0^t e^{bs}B_{H,\lambda}(s)dsdZ_t &= \lim_{T\to\infty} \frac{\int_0^T \int_0^t e^{bs} B_{H,\lambda}(s)ds e^{-bt}B_{H,\lambda}(t)dt}{e^{bT}} \\
&=\lim_{T\to\infty} \frac{\int_0^T e^{bs} B_{H,\lambda}(s)ds e^{-bT}B_{H,\lambda}(t)}{be^{bT}} = 0 \quad \text{a.s.}
\end{align*}
in view of \eqref{eq:tfbm-growth}.

Hence, 
\begin{equation}\label{eq:RT3}
R''_{3,T} \to 0 \quad \text{a.s., as } T \to\infty.
\end{equation}

Combining \eqref{eq:FT}, \eqref{eq:FT1_re2}, \eqref{eq:FT2_re2}, \eqref{eq:FT3_re2}, \eqref{eq:FT32_re2}, \eqref{eq:F33_re2}, and \eqref{eq:F34T_re2}, we arrive at
\begin{equation}\label{eq:FTres}
F_T = \sigma V_T (y_0 + \frac{a}{b}+b\sigma Z_T)+ \widetilde{R}_T,
\end{equation}
where 
\begin{equation}\label{eq:RTwidetilde}
\widetilde{R}_T = R_{1,T}+R_{2,T}+R'_{3,T}+R''_{3,T} \to 0 \quad \text{a.s., as } T \to\infty,
\end{equation}
by  \eqref{eq:RT2}, \eqref{eq:RT2_2}, \eqref{eq:RT3_re2}, and \eqref{eq:RT3}.

We complete the proof by inserting \eqref{eq:FTres} into \eqref{eq:int-b-asymp} and noticing that 
$R_T \coloneqq \frac{\widetilde{R}_T}{D_T} \to 0$
a.s., as  $T \to\infty$ in view of \eqref{eq:RTwidetilde} and \eqref{eq:Dt}.
\end{proof}

In the next lemma, we express the estimator $\hat a_T$ via the integrals $U_T$ and $V_T$ defined in \eqref{eq:ZtUt}--\eqref{eq:Vt}. This representation also contains random variables $P_T$ and $Q_T$, converging a.s.\ to the constants 1 and 0 respectively.
\begin{lemma}\label{l:delta-a}
For all $T>0$
\[
T (\hat a_T - a) = b\sigma U_T - \sigma V_T P_T + Q_T,
\]   
where
\begin{equation}\label{eq:PT}
P_T \coloneqq  \frac{\left(y_0 + \frac{a}{b} + b\sigma Z_T\right) e^{-bT}\int_0^T Y_t dt}{D_T} - 1 \to 1 
\quad \text{a.s., as } T \to\infty.
\end{equation}
and 
\begin{equation}\label{eq:QT}
Q_T \coloneqq - R_T e^{-bT}\int_0^T Y_t dt \to 0 \quad \text{a.s., as } T \to\infty.
\end{equation}
\end{lemma}

\begin{proof}
Using \eqref{eq:a-est} and \eqref{eq:b-est} we rewrite $T \hat a_T$ as follows
\begin{align*}
T \hat a_T 
&= \frac{\left(Y_T - y_0\right) \left(T\int_0^T Y_t^2 dt - \left(\int_0^T Y_t dt\right)^2 + \left(\int_0^T Y_t dt\right)^2 -  \frac12 T \left(Y_T+y_0\right) \int_0^T Y_t dt\right)}{T \int_0^T Y_t^2 dt - \left(\int_0^T Y_t dt\right)^2}
\\
&= Y_T - y_0 - \hat b_T \int_0^T Y_t dt.
\end{align*}  
Now expressing $Y_T$ through \eqref{Main-SDE}, we get
\begin{equation}\label{eq:Ta}
T \hat a_T = T a + \sigma B_{H,\lambda}(T) - \left(\hat b_T - b\right)\int_0^T Y_t dt.
\end{equation}
Note that by \eqref{eq:Vt},
$B_{H,\lambda}(T) = b U_T + V_T$.
Using this relation and the representation \eqref{eq:repr-b} we derive from \eqref{eq:Ta} that
\begin{align*}
T \left(\hat a_T - a\right) &= \sigma V_T + b\sigma U_T - \left(\frac{\sigma V_T \left(y_0 + \frac{a}{b} + b\sigma Z_T\right)}{D_T} + R_T \right) e^{-bT}\int_0^T Y_t dt
\\
&= b\sigma U_T - \sigma V_T P_T
+ Q_T.
\end{align*}
Note that
$y_0 + \frac{a}{b} + b\sigma Z_T \to \zeta$,
$D_T \to \frac{1}{2b}\zeta^2$,
$e^{-bT}\int_0^T Y_t dt \to  \frac{1}{b}\zeta$
and $R_T \to 0$
a.s., as $T \to\infty$,
by \eqref{eq:zeta}, \eqref{eq:Dt}, \eqref{eq:int-Y-asymp}, and \eqref{eq:RT} respectively.
This implies the convergences
\eqref{eq:PT} and \eqref{eq:QT}.
\end{proof}

\subsection{Asymptotic normality of \texorpdfstring{$(Z_T, U_T, V_T)$}{(Z, U, V)}}
\label{ssec:as-norm-aux}
The purpose of this subsection is to find a joint asymptotic distribution of the integrals $Z_T$, $U_T$, and $V_T$ as $T\to\infty$.
This distribution is obviously Gaussian, since $Z_T$, $U_T$, and $V_T$ are Gaussian processes.
Therefore, it suffices to calculate the elements of the asymptotic covariance matrix.
This will be done in the following series of lemmas.
The limits contain the constants $\alpha_{H,\lambda}$ and 
$\beta_{H,\lambda,b}$ defined in Theorem~\ref{th:main}.

\begin{lemma}\label{l:var-Z}
The following convergence holds:
\begin{equation}\label{eq:var-Z-asymp}
\lim_{T \to  \infty} \E Z_T^2 = \E Z_\infty^2 = \frac{\beta_{H,\lambda,b}^2}{b^2}.
\end{equation}
\end{lemma}

\begin{proof}
Using the definition \eqref{eq:Zinf} of $Z_\infty$, and the formula \eqref{eq:cov} for the covariance function of TFBM, we may write
\begin{align}
\E Z_\infty^2 &= \E \left(\int_0^\infty e^{-bt} B_{H,\lambda}(t)dt\right)^2
= \int_0^\infty\!\!\int_0^\infty e^{-bt-bs} \E \left[B_{H,\lambda}(t)B_{H,\lambda}(s)\right]dt\,ds
\notag\\
&= \frac12 \int_0^\infty\!\!\int_0^\infty e^{-bt-bs} \left(C_{t}^2 t^{2H}+ C_{s}^2 s^{2H} - C_{\abs{t-s}}^2 \abs {t-s}^{2H}\right)ds\,dt
\notag\\
&= \int_0^\infty e^{-bs} ds \int_0^\infty e^{-bt} C_{t}^2 t^{2H}dt
- \frac12 \int_0^\infty\!\!\int_0^\infty e^{-bt-bs} C_{\abs{t-s}}^2 \abs {t-s}^{2H}\,ds\,dt
\notag\\
&\eqqcolon A_1 + A_2.
\label{eq:A1+A2}
\end{align}

Since $\int_0^\infty e^{-bs} ds = \frac1b$, we see that
\begin{equation}\label{eq:A1}
A_1 = \frac1b \int_0^\infty e^{-bt} C_{t}^2 t^{2H}dt
= \frac{2}{b^2}\, \beta_{H,\lambda,b}^2
\end{equation}
by definition of $\beta_{H,\lambda,b}$, see \eqref{eq:alphabeta}.

Let us consider $A_2$. Due to the symmetry of the integrand, we have
\[
A_2 = - \int_0^\infty\!\!\int_0^t e^{-bt-bs} C_{t-s}^2 (t - s)^{2H}\,ds\,dt
=  - \int_0^\infty\!\!\int_0^t e^{-2bt + bu} C_u^2 u^{2H}\,du\,dt,
\]
where we have used the substitution $u = t - s$ in the inner integral.
Changing the order of integration and integrating w.r.t.\ $t$, we then get
\begin{equation}\label{eq:A2}
A_2 =  - \int_0^\infty \left(\int_u^\infty e^{-2bt} dt \right) e^{bu} C_u^2 u^{2H}\,du
= - \frac{1}{2b} \int_0^\infty e^{-bu} C_u^2 u^{2H}\,du 
= - \frac{1}{b^2} \beta_{H,\lambda,b}^2.
\end{equation}
Combining \eqref{eq:A1+A2}--\eqref{eq:A2}, we obtain \eqref{eq:var-Z-asymp}.
\end{proof}

\begin{lemma} \label{l:var-U}
The following convergence holds:
\begin{equation}\label{eq:var-U-asymp}
\lim_{T \to  \infty} \E U_T^2 = \frac{1}{b^2}\left( \alpha_{H,\lambda}^2-\beta_{H,\lambda,b}^2 \right).
\end{equation}
\end{lemma}

\begin{proof}
Using the definition \eqref{eq:ZtUt} of $U_T$ and the formula \eqref{eq:cov} for the covariance function of TFBM, we have
\begin{align}
\E U_T^2 &=
\E \left(\exp\{-b T\}\int_{0}^{T}\exp \left\{ b s \right\} B_{H,\lambda} (s)ds\right)^2
\notag\\
&= \frac{1}{2} \exp\{-2b T\}\int_{0}^{T}\!\!\int_{0}^{T}\exp \{ b (s+t)  \}\left[C^2_{t} t^{2H}+C^2_{s} s^{2H}  - C^2_{\abs{t-s}} \abs {t-s}^{2H}\right]ds\,dt
\notag\\
&= \exp\{-2b T\}  \int_{0}^{T}\exp \{ b s \}ds \int_{0}^{T}\exp \{ b t \}C^2_t t^{2H}\,dt
\notag\\
&\quad-\frac{1}{2} \exp\{-2b T\}\int_{0}^{T}\!\! \int_{0}^{T} \exp\{b (s+t) \}C^2_{\abs {t-s}} \abs {t-s}^{2H}ds\,dt
\notag\\
&\eqqcolon B_{1,T} + B_{2,T}.
\label{eq:B1T+B2T}
\end{align}
By the l'H\^opital rule and \eqref{eq:TFBMI-asymp}, we have
\begin{equation}\label{eq:lim-integral}
\lim_{T\to\infty} \frac{\int_{0}^{T} e^{bt} C^2_t t^{2H} dt}{e^{bT}}
= \lim_{T\to\infty} \frac{e^{bT} C^2_T T^{2H}}{b e^{bT}}
= \frac{\alpha_{H,\lambda}^2}{b},
\end{equation}
whence
\begin{equation}\label{eq:B1T}
\lim_{T\to\infty} B_{1,T}
= \lim_{T\to\infty} \frac{1 - e^{-b T}}{b} \cdot \frac{\int_{0}^{T} e^{bt} C^2_t t^{2H} dt}{e^{bT}}
= \frac{\alpha_{H,\lambda}^2}{b^2}.
\end{equation}

Taking into account the symmetry of the integrand, we can represent $B_{2,T}$ in the following form:
\begin{align*}
B_{2,T} &= - \exp\{-2b T\} \int_{0}^{T} \!\!\int_{0}^{t}\exp\{b (s+t) \} C^2_{t-s} (t-s)^{2H}ds\,dt
\\
&= -\exp\{-2b T\}\int_{0}^{T}\!\! \int_{0}^{t}\exp\{b (2t-u) \} C^2_{u} u^{2H}du\,dt,
\end{align*}
where we have used the substitution
$u = t-s$ in the inner integral.
Changing the order of integration and integrating w.r.t.\ $t$,
we obtain
\begin{align*}
B_{2,T} &= - \exp\{-2b T\}\int_{0}^{T} \exp\{-b u \} C^2_{u} u^{2H}\int_{u}^{T}\exp\{2b t \}dt\,du
\\
&= -\frac{1}{2b} \int_{0}^{T} \exp\{-b u \} C^2_{u} u^{2H}du
+ \frac{1}{2b} \exp\{-2b T\}\int_{0}^{T} \exp\{b u \} C^2_{u} u^{2H}du.
\end{align*}
Note that the last term in the right-hand side of the above equality tends to zero due to \eqref{eq:lim-integral}. Therefore
\begin{equation}\label{eq:B2T}
\lim_{T\to\infty} B_{2,T} = -\frac{1}{2b} \int_{0}^{\infty} \exp\{-b u \} C^2_{u} u^{2H}du 
= - \frac{1}{b^2} \beta_{H,\lambda,b}^2,
\end{equation}
by the definition of $\beta_{H,\lambda,b}$, see \eqref{eq:alphabeta}.
Combining \eqref{eq:B1T+B2T}, \eqref{eq:B1T}, and \eqref{eq:B2T}, we conclude the proof.
\end{proof}

\begin{lemma}\label{l:A3}
The following convergence holds:
\begin{equation}\label{eq:A3}
\lim_{T \to  \infty} \E V_T^2 = \beta_{H,\lambda,b}^2.
\end{equation}
\end{lemma}

\begin{proof}
Using \eqref{eq:Vt}, we represent the left-hand side of \eqref{eq:A3} in the following form
\begin{equation}\label{eq:var-V-decomp}
\E V_T^2 =
\E \left(B_{H,\lambda}(T) - b U_T\right)^2
=\E B_{H,\lambda}^2(T) + b^2 \E U_T^2 - 2b \E \left[B_{H,\lambda} (T) U_T\right].
\end{equation}
Next, we transform the third term in the right-hand side of \eqref{eq:var-V-decomp} as follows
\begin{align}
\MoveEqLeft
\E \left[B_{H,\lambda} (T) U_T\right]
=
\exp\{-b T\}\int_{0}^{T}\exp \left\{ b s \right\} \E [ B_{H,\lambda} (T) B_{H,\lambda} (s)]ds 
\notag\\
&= \frac12 \exp\{-b T\}\int_{0}^{T}\exp \left\{ b s \right\} \left[C^2_{T} T^{2H}+C^2_{s} s^{2H}  - C^2_{T-s} (T-s)^{2H}\right] ds
\notag\\
&= \frac12 \exp\{-b T\}C^2_{T} T^{2H}\int_{0}^{T}\exp \left\{ b s \right\} ds
+\frac12\exp\{-b T\}\int_{0}^{T}\exp \left\{ b s \right\} C^2_{s} s^{2H} ds
\notag\\
&\quad - \frac12 \exp\{-b T\}\int_{0}^{T}\exp \left\{ b s \right\}  C^2_{T-s} (T-s)^{2H}ds
\notag\\
&= \frac{1}{2b} C^2_{T} T^{2H} - \frac{1}{2b}\exp\{-b T\}C^2_{T} T^{2H}
+ \frac12 \exp\{-b T\}\int_{0}^{T}\exp \left\{ b s \right\} C^2_{s} s^{2H} ds
\notag\\
&\quad - \frac12 \int_{0}^{T}\exp \left\{ - b z \right\}  C^2_{z} z^{2H}dz.
\label{eq:covBU}
\end{align}
Combining this expression with \eqref{eq:var-V-decomp} and taking into account that $C^2_{T} T^{2H} = \E [B_{H,\lambda}(T)^2]$ we arrive at
\begin{equation}\label{eq:var-V-decomp2}
\begin{split}
\E V_T^2 & = 
b^2 \E U_T^2 + \exp\{-b T\}C^2_{T} T^{2H}
-b \exp\{-b T\}\int_{0}^{T}\exp \left\{ b s \right\} C^2_{s} s^{2H} ds
\\
&\quad +b \int_{0}^{T}\exp \left\{ - b z \right\}  C^2_{z} z^{2H}dz.
\end{split}
\end{equation}
Note that according to \eqref{eq:TFBMI-asymp} the second term in the right-hand side of \eqref{eq:var-V-decomp2} vanishes as $T \to \infty$, while the limits of other terms are already known, see \eqref{eq:var-U-asymp}, \eqref{eq:lim-integral}, and \eqref{eq:alphabeta}.
Therefore, we arrive at
\[
\lim_{T\to\infty} \E V_T^2  = 
\left( \alpha_{H,\lambda}^2-\beta_{H,\lambda,b}^2 \right)
- \alpha_{H,\lambda}^2 + 2 \beta_{H,\lambda,b}^2 = \beta_{H,\lambda,b}^2.
\qedhere
\]
\end{proof}

\begin{lemma}\label{l:cov-ZU} The 
  following asymptotics holds 
\begin{equation*}
\lim_{T \to  \infty} \E \left[Z_T U_T\right] = \frac{\beta_{H,\lambda,b}^2}{b^2}. 
\end{equation*}
\end{lemma}

\begin{proof}
Using formula \eqref{eq:cov} for the covariance function of TFBM, we get 
\begin{align*}
\E \left[Z_T U_T\right] &= e^{-bT} \E \left[\int_0^{T}e^{bt}B_{H,\lambda}(t)dt \int_0^{T}e^{-bs}B_{H,\lambda}(s)ds \right] \\
&=\frac{1}{2}  e^{-bT} \int_0^{T}\!\! \int_0^{T} e^{bt-bs} \left(C_t^2 t^{2H} + C_s^2 s^{2H} - C_{\abs{t-s}}^2 \abs{t-s}^{2H} \right)ds\,dt \\
&=I_{1,T}+I_{2,T}+I_{3,T}+I_{4,T},
\end{align*}
where
\begin{align*}
I_{1,T} =  \frac{1}{2}e^{-bT}\int_0^T e^{bt} C^2_{t} t^{2H}dt \int_0^{T}e^{-bs}ds,
\end{align*}
\begin{align*}
I_{2,T} =  \frac{1}{2}e^{-bT}\int_0^T e^{-bs} C^2_{s} s^{2H}ds \int_0^{T}e^{bt}dt,
\end{align*}
\begin{align*}
I_{3,T} =  -\frac{1}{2}e^{-bT}\int_0^T\!\! \int_0^{t} e^{bt-bs} C^2_{t-s} (t-s)^{2H}ds\,dt,
\end{align*}
\begin{align*}
I_{4,T} =  -\frac{1}{2}e^{-bT}\int_0^T\!\! \int_t^{T} e^{bt-bs} C^2_{s-t} (s-t)^{2H}ds\,dt.
\end{align*}
By \eqref{eq:lim-integral},
\begin{align*}
I_{1,T} =  \frac{1}{2}e^{-bT}\cdot\frac{1-e^{-bT}}{b}\int_0^T e^{bt} C^2_{t} t^{2H}dt \to \frac{\alpha^2_{H,\lambda}}{2b^2},
\quad\text{as } T \to \infty.
\end{align*}
Taking into account \eqref{eq:alphabeta}, we get
\begin{align*}
I_{2,T} =  \frac{1}{2}e^{-bT}\cdot\frac{e^{-bT}-1}{b}\int_0^T e^{-bs} C^2_{s} s^{2H}ds \to \frac{\beta^2_{H,\lambda,\beta}}{b^2},
\quad\text{as } T \to \infty.
\end{align*}
By l'Hopital's rule and \eqref{eq:lim-integral}, we have
\begin{align*}
I_{3,T} =  -\lim_{T \to  \infty} \frac{\int_0^T \!\int_0^t e^{bu}C^2_u u^{2u}du}{2e^{bT}} = -\lim_{T \to  \infty} \frac{\int_0^T e^{bu}C^2_u u^{2u}du}{2be^{bT}} = -\frac{\alpha^2_{H,\lambda}}{2b^2}.
\end{align*}
Changing the order of integration, we obtain
\begin{align*}
I_{4,T} =  -\frac{1}{2}e^{-bT}\int_0^T\!\! \int_0^{s} e^{bt-bs} C^2_{s-t} (s-t)^{2H}dt\,ds =  -\frac{1}{2}e^{-bT}\int_0^T\!\! \int_0^{s} e^{-bu} C^2_{u} u^{2H}du\,ds.
\end{align*}
By l`Hopital's rule and \eqref{eq:alphabeta}, 
\begin{align*}
\lim_{T \to  \infty}I_{4,T} = -\lim_{T \to  \infty} \frac{\int_0^T e^{-bu}C^2_u u^{2u}du}{2be^{bT}} = 0.
\end{align*}
Collecting all the limits completes the proof.
\end{proof}

\begin{lemma} The next value is asymptotically negligible:
\begin{equation*}
\lim_{T \to  \infty} \E \left[Z_T V_T\right] = 0. 
\end{equation*}
\end{lemma}

\begin{proof}
It follows from \eqref{eq:Vt}  that
\begin{equation}\label{l:E_ZV}
\E \left[Z_T V_T\right] = \E \left[Z_T B_{H,\lambda}(T)\right]-b\E \left[Z_T U_T\right]. 
\end{equation}
By \eqref{eq:TFBMI-asymp}, \eqref{eq:alphabeta}, and \eqref{eq:lim-integral}, we obtain the next equalities 
\begin{align}
\E \left[Z_T B_{H,\lambda}(T)\right] &= \E \left[B_{H,\lambda}(T) \int_0^T e^{-bt} B_{H,\lambda}(t)dt\right] 
\notag\\
&=\frac{1}{2} \int_0^{T}e^{-bT}\left(C_T^2 T^{2H} + C_t^2 t^{2H} - C_{T-t}^2 (T-t)^{2H} \right)dt 
\notag\\
&=\frac{1}{2} C_T^2 T^{2H}\cdot\frac{1-e^{-bT}}{b}+\frac{1}{2}\int_0^T e^{-bt} C^2_{t} t^{2H}dt 
\notag\\
&\quad-\frac{1}{2}\int_0^T e^{-b(T-s)} C^2_{s} s^{2H}ds
\notag\\
&\to \frac{\alpha^2_{H,\lambda}}{2b}  + \frac{\beta^2_{H,\lambda,\beta}}{b} - \frac{\alpha^2_{H,\lambda}}{2b} 
=\frac{\beta^2_{H,\lambda,\beta}}{b}, 
\quad \text{as } T \to \infty.
\label{l:E_ZB}
\end{align}
Combining \eqref{l:E_ZV}, \eqref{l:E_ZB} and Lemma \ref{l:cov-ZU}, we conclude the proof.
\end{proof}

\begin{lemma}\label{l:cov-UV} The following value is also negligible: 
\begin{equation*}
\lim_{T \to  \infty} \E \left[U_T V_T\right] = 0. 
\end{equation*}
\end{lemma}

\begin{proof}
From the representation \eqref{eq:covBU} we derive using \eqref{eq:TFBMI-asymp}, \eqref{eq:lim-integral}, and \eqref{eq:alphabeta} that
\begin{equation}\label{eq:covBU-1}
\lim_{T\to\infty}  \E \left[ B_{H,\lambda} U_T\right]
= \frac{\alpha_{H,\lambda}^2}{2b} + 0 + \frac{\alpha_{H,\lambda}^2}{2b} - \frac{\beta_{H,\lambda,b}^2}{b}
=\frac{\alpha_{H,\lambda}^2 - \beta_{H,\lambda,b}^2}{b}
= b \lim_{T \to  \infty} \E U_T^2,
\end{equation}
where the last equality follows from Lemma~\ref{l:var-U}.
Taking into account that
by \eqref{eq:Vt},
\[
\E \left[ U_T V_T \right] = \E \left[ B_{H,\lambda} U_T\right] - b \E U_T^2,
\]
we complete the proof.
\end{proof}

\begin{proposition}\label{prop:ZUV-asymp}
As $T\to\infty$,
\[
\begin{pmatrix}
Z_T \\ U_T \\ V_T
\end{pmatrix}
\dto \ND (\mathbf 0, \Sigma),
\]
where
\[
\Sigma = 
\begin{pmatrix}
b^{-2}\beta_{H,\lambda,b}^2 & b^{-2}\beta_{H,\lambda,b}^2 & 0 \\
b^{-2}\beta_{H,\lambda,b}^2 & b^{-2}\left( \alpha_{H,\lambda}^2-\beta_{H,\lambda,b}^2\right) & 0\\
0 & 0 & \beta_{H,\lambda,b}^2
\end{pmatrix}.
\]
In particular,
\begin{itemize}
    \item $Z_T$ and $V_T$ are asymptotically independent;
    \item $U_T$ and $V_T$ are asymptotically independent.
\end{itemize}
\end{proposition}

\begin{proof}
Lemmas \ref{l:var-Z}--\ref{l:cov-UV} together give us the value of the asymptotic covariance matrix. Taking into account the normality of the random vector $(Z_T, V_T, U_T)^\top$, we obtain the desired convergence.
\end{proof}

\subsection{Proof of Theorem \ref{th:main} and convergence (\ref{eq:joint})}
\label{ssec:proof-main}
Let $(\xi_1, \xi_2, \xi_3) \simeq\ND(\mathbf 0, \Sigma)$, where the matrix $\Sigma$ is defined in Proposition \ref{prop:ZUV-asymp}.

$(i)$
By Lemma~\ref{l:delta-a} and Proposition~\ref{prop:ZUV-asymp}, 
\[
T \left(\hat a_T - a\right)
= b\sigma U_T - \sigma V_T P_T + Q_T
\dto b\sigma \xi_2 - \sigma \xi_3,
\quad\text{as } T \to \infty,
\]
where $\xi_2 \simeq \ND (0, b^{-2}(\alpha_{H,\lambda}^2-\beta_{H,\lambda,b}^2))$ and $\xi_3 \simeq \ND (0,\beta_{H,\lambda,b}^2)$ are uncorrelated (hence, independent) Gaussian random variables.
Therefore, the limiting distribution $b\sigma \xi_2 - \sigma \xi_3$ is zero-mean Gaussian with variance
\[
b^2\sigma^2 b^{-2} (\alpha_{H,\lambda}^2-\beta_{H,\lambda,b}^2) + \sigma^2 \beta_{H,\lambda,b}^2
=\sigma^2 \alpha_{H,\lambda}^2.
\]
Thus, \eqref{eq:a-asnorm} is proved.

$(ii)$
According to Lemma~\ref{l:delta-b}, one has the following representation
\[
e^{bt} \left(\hat b_T - b\right) 
=
\frac{\left(y_0 + \frac{a}{b} + b\sigma Z_T\right)^2}{D_T} \frac{\sigma V_T}{y_0 + \frac{a}{b} + b\sigma Z_T} + R_T,
\]
where $R_T\to0$ a.s.\ when $T\to\infty$.
By \eqref{eq:zeta} and \eqref{eq:Dt},
\[
\frac{\left(y_0 + \frac{a}{b} + b\sigma Z_T\right)^2}{D_T}
\to \frac{\zeta^2}{\frac{1}{2b}\zeta^2} = 2b
\quad \text{a.s., as } T \to\infty.
\]
Therefore, we derive from Proposition \ref{prop:ZUV-asymp} and the Slutsky theorem that
\[
e^{bt} \left(\hat b_T - b\right) 
\dto
\frac{2b\sigma \xi_3}{y_0 + \frac{a}{b} + b\sigma \xi_1}
= \frac{\eta_1}{\eta_2},
\]
where
$\eta_1 \coloneqq 2b\sigma \xi_3 \simeq \ND(0, 4 b^2 \sigma^2 \beta_{H,\lambda,b}^2)$
and
$\eta_2 \coloneqq y_0 + \frac{a}{b} + b\sigma \xi_1
\simeq \ND(y_0 + \frac{a}{b}, \sigma^2 \beta_{H,\lambda,b}^2)$
are uncorrelated, hence, independent.
\qed

\appendix
  \section{Asymptotic behavior of drift parameters estimator for the Vasicek model driven by Gaussian process}
  \label{app:es-sebaiy}

  In this appendix we formulate the main result of \cite{Es-Sebaiy} concerning the asymptotic distribution of the parameter estimators in the Gaussian Vasicek-type model and give the comments which conditions of this paper are satisfied and which are not, therefore it was necessary to  modify the respective proofs. 
  So, let $G:=\left\{G_{t}, t \geq 0\right\}$ be a centered Gaussian process satisfying the following assumption
  \begin{itemize}
  \item[$(\mathcal{A}_{1})$] There exist constants $c>0$ and $\gamma \in(0,1)$ such that for every $s, t \geq 0$,
  \[
  G_{0}=0, \quad \E\left[\left(G_{t}-G_{s}\right)^{2}\right] \leq c|t-s|^{2 \gamma}.
  \]
  \end{itemize}
  The Gaussian Vasicek-type process $X=\left\{X_{t}, t \geq 0\right\}$ is defined as the unique (pathwise) solution to
  \begin{equation*}
  X_{0}=0, \quad d X_{t} = \left(a + b X_{t}\right) d t + d G_{t}, \quad t \geq 0.
  \end{equation*}
  where $a \in \mathbb{R}$ and $b>0$ are considered as unknown parameters. The corresponding least-squares estimators have the form
  \begin{equation*}
  \tilde{b}_{T}=\frac{\frac{1}{2} T X_{T}^{2}-X_{T} \int_{0}^{T} X_{s} d s}{T \int_{0}^{T} X_{s}^{2} d s-\left(\int_{0}^{T} X_{s} d s\right)^{2}}
  \quad\text{and}\quad
  \widetilde{a}_{T}=\frac{X_{T} \int_{0}^{T} X_{s}^{2} d s-\frac{1}{2} X_{T}^{2} \int_{0}^{T} X_{s} d s}{T \int_{0}^{T} X_{s}^{2} d s-\left(\int_{0}^{T} X_{s} d s\right)^{2}}.
  \end{equation*}

  The following additional assumptions are required:
 \begin{itemize}
 \item[$(\mathcal{A}_{2})$]
  There exist $\lambda_{G}>0$ and $\eta \in(0,1)$ such that, as $T \rightarrow \infty$
  \begin{equation}\label{eq:condA2}
 \frac{\E\left(G_{T}^{2}\right)}{T^{2 \eta}} \rightarrow \lambda_{G}^{2}.
  \end{equation}

 \item[$(\mathcal{A}_{3})$]
  There exists a constant $\sigma_{G}>0$ such that
  \[
  \lim _{T \rightarrow \infty} \E\left[\left(e^{-b T} \int_{0}^{T} e^{b s} d G_{s}\right)^{2}\right] = \sigma_{G}^{2}.
  \]

 \item[$(\mathcal{A}_{4})$]
 \[
  \lim _{T \rightarrow \infty} \E\left(G_{s} e^{-b T} \int_{0}^{T} e^{b r} d G_{r}\right)=0.
  \]

 \item[$(\mathcal{A}_{5})$]
  For all fixed $s \geq 0$,
  \[
 \lim _{T \rightarrow \infty} \frac{\E\left(G_{s} G_{T}\right)}{T^{\eta}}=0, \quad \lim _{T \rightarrow \infty} \E\left(\frac{G_{T}}{T^{\eta}} e^{-b T} \int_{0}^{T} e^{b r} d G_{r}\right)=0.
 \]
  \end{itemize}

  \begin{theorem}[{\cite[Theorem 3.2]{Es-Sebaiy}}]\label{es-es}
 Assume that
  $\left(\mathcal{A}_{1}\right)$--$\left(\mathcal{A}_{4}\right)$ hold. Suppose that $N_{1} \sim \mathcal{N}(0,1)$, $N_{2} \sim \mathcal{N}(0,1)$ and $G$ are independent. Then as $T \rightarrow \infty$,
  \begin{gather*}
  e^{b T}\left(\widetilde{b}_{T}-b\right) \dto \frac{2 b \sigma_{G} N_{2}}{\mu+\zeta_{\infty}},\\
  T^{1-\eta}\left(\widetilde{a}_{T}-a\right) \dto \lambda_{G} N_{1}.
  \end{gather*}
  Moreover, if $\left(\mathcal{A}_{5}\right)$ holds, then as $T \rightarrow \infty$,
  \begin{align*}
  \left(e^{b T}\left(\widetilde{b}_{T}-b\right), T^{1-\eta}\left(\widetilde{a}_{T}-a\right)\right) \dto \left(\frac{2 b \sigma_{G} N_{2}}{\mu+\zeta_{\infty}}, \lambda_{G} N_{1}\right).
  \end{align*}
  \end{theorem}

Let us analyze whether the conditions $(\mathcal{A}_{1})$--$(\mathcal{A}_{5})$ are satisfied for our tempered fractional Vasicek model \eqref{Main-SDE}.

We start with the basic condition $(\mathcal{A}_{1})$.
The behavior of the variogram function $\E (B_{H,\lambda}(t) - B_{H,\lambda}(s))^2$ of TFBM was recently studied in \cite{MR24}, where the following upper bounds have been established (see \cite[Lemma~2]{MR24}):
\begin{enumerate}[label=$(\roman*)$]
\item If $H\in(0,1)$, then for all $t, s \in \real_+$
\[
 \E \abs{B_{H,\lambda}(t) - B_{H,\lambda}(s)}^2\leq C \left(\abs{t - s}^{2H}\wedge 1\right).
\]

\item If $H = 1$, then for all $t, s \in \real_+$
\[
 \E \abs{B_{H,\lambda}(t) - B_{H,\lambda}(s)}^2 \leq C \left(\abs{t - s}^2 \bigl\lvert\log\abs{t-s}\bigr\rvert\wedge 1\right).
\]

\item If $H > 1$, then for all $t, s \in \real_+$
\[
 \E \abs{B_{H,\lambda}(t) - B_{H,\lambda}(s)}^2 \leq C (\abs{t - s}^{2}\wedge 1).
\]
\end{enumerate}
Comparison of these bounds with the assumption $(\mathcal{A}_{1})$ shows that this assumption holds only for $H\in(0,1)$. 
For $H>1$ one should choose $\gamma = 1$ in this assumption, which is impossible. More careful analysis of the proofs in \cite{Es-Sebaiy} shows that the condition $\gamma < 1$ is substantial for \cite[Lemma 2.2]{Es-Sebaiy}, which provides the almost sure convergences
\begin{equation}\label{eq:es-l22}
\frac{G_T}{T_\delta} \to 0,
\quad
\frac{e^{-b t}}{T} \int_0^T \abs{G_t X_t} dt \to 0,
\quad T \to \infty,
\end{equation}
for any $\gamma < \delta \le 1$.
Based on this convergence, the authors of \cite{Es-Sebaiy} derive the convergences of the form \eqref{eq:Y-asymp}--\eqref{eq:int-Yt-asymp}, on which the subsequent study of the estimators is based.
Thus, we cannot apply the results of \cite{Es-Sebaiy} directly to the case $G = B_{H,\lambda}$ when $H\ge1$.
However, the almost sure upper bound~\eqref{eq:tfbm-growth} allows us to obtain the first convergence in \eqref{eq:es-l22} for any $\delta > 0$. This bound also makes it possible to derive the second convergence (see the proof of \cite[Lemma 7]{MR24}) as well as the convergences  \eqref{eq:Y-asymp}--\eqref{eq:int-Yt-asymp} \cite[Lemma 6]{MR24}, which in turn lead to the strong consistency of the estimators  \cite[Theorem 6]{MR24}.

Furthermore, for the case $G = B_{H,\lambda}$, the condition $(\mathcal{A}_2)$ is also violated (for any $H>0$). Indeed, in view of \eqref{eq:TFBMI-asymp}, the convergence \eqref{eq:condA2} in $(\mathcal{A}_2)$ holds with $\eta = 0$ instead of $\eta\in(0,1)$. This affects on the behaviour of the estimator $\hat a_T$, which has the following representation:
\[
T^{1-\eta} (\tilde a_T - a) =  - \frac{1}{T^\eta} e^{b T}\left(\tilde b_T - b\right)  e^{-bT}\int_0^T X_t dt + \frac{G_T}{T^\eta}.
\]   
If $\eta > 0$, then the first term of this representation vanishes as $T \to \infty$, and the second one converges to a normal distribution $\ND (0,\lambda_G^2)$. If $\eta = 0$, then both terms have non-trivial limits (in fact, they both converge to normal distributions); hence, the study of the asymptotic behavior of their sum becomes more involved.

The conditions $(\mathcal{A}_3)$ and $(\mathcal{A}_4)$ are satisfied when $G$ is a TFBM. Namely, $(\mathcal{A}_3)$  is verified in Lemma~\ref{l:A3}, and $(\mathcal{A}_4)$ can be checked in a similar way.

Finally, the condition $(\mathcal{A}_5)$ does not hold in the case of TFBM (for all $H>0$).
In particular, for any $s > 0$,
\[
\lim_{T\to\infty} \E \left[B_{H,\lambda} (s) B_{H,\lambda} (T)\right]
= \frac12 C^2_s s^{2H} \ne 0,
\]
and, moreover,
\begin{multline*}
\lim_{T \rightarrow \infty} \E\left[B_{H,\lambda} (T) e^{-b T} \int_{0}^{T} e^{b r} d B_{H,\lambda} (r)\right]
= \lim_{T \rightarrow \infty} \E\left[B_{H,\lambda} (T) V_T\right]
\\
= \lim_{T \rightarrow \infty} \left(\E B_{H,\lambda}^2(T) - b\E\left[B_{H,\lambda} (T) U_T\right]\right)
= \beta_{H,\lambda,b}^2
\ne0.
\end{multline*}
where the limit is computed by \eqref{eq:TFBMI-asymp} and \eqref{eq:covBU-1}.
Thus, both equalities in $(\mathcal{A}_5)$ are not valid.
Consequently, the asymptotic independence of the estimators is not guaranteed in the case of TFBM. In Remark~\ref{rem:joint} we explain the correlation between estimators in more detail.

Additionally, compared to \cite{Es-Sebaiy}, we do not restrict ourselves with zero initial condition allowing it to be any non-random constant $Y_0 = y_0 \in \real$.

\section{Special functions \texorpdfstring{$K_\nu$ and $\hyper$}{K and F}}
\label{app:special}

In this appendix, we present the definitions of the function $K_{\nu}$, which appears in the representation of the covariance function of TFBM (see \eqref{eq:cov}--\eqref{eq:Ct}), and the function $\hyper$ from the representation \eqref{eq:repr-beta-hyp} for the constant $\beta_{H,\lambda,b}^2$.
For further information on this topic, we refer to the book \cite{Andrews1999}.

The \emph{modified Bessel function of the second kind} $K_{\nu}(x)$ has the integral representation
\begin{equation*}
K_{\nu}(x)=\int_0^\infty e^{-x \cosh t} \cosh {\nu t}\ dt,
\end{equation*}
where $\nu>0$, $x>0$. The function $K_{\nu}(x)$ also has the series representation
\begin{equation*}
K_{\nu}(x) = \frac{\pi}{2}\, \frac{ I_{-\nu}(x) - I_{\nu}(x) }{\sin(\pi \nu)},
\end{equation*}
where $I_{\nu}(x)=(\frac{1}{2}|x|)^{\nu} \sum_{n=0}^{\infty} \frac{ ( \frac{1}{2}x)^{2n}  }{n! \Gamma(n+1+\nu)}$ is called the \emph{modified Bessel function of the first kind}.

The \emph{Gauss hypergeometric function} $\hyper(a,b;c;x)$ can be defined for complex $a$, $b$, $c$ and $x$. Here,
we restrict ourselves to the case of real arguments.
Moreover, we assume that $c>b>0$.
In this case, we may define $\hyper(a,b;c;x)$ for $x<1$ by
the following Euler's integral representation \cite[Theorem~2.2.1]{Andrews1999}:
\begin{equation*}
\hyper(a,b;c;x) = \frac{\Gamma(c)}{\Gamma(b)\Gamma(c-b)}
\int_0^1 t^{b-1} (1-t)^{c-b-1} (1-xt)^{-a} \,dt.
\end{equation*}

\bibliographystyle{abbrv}
\bibliography{bibtempered}
\end{document}